\documentclass[12pt]{article}
\usepackage[margin=1in]{geometry}
\usepackage{amsmath,amssymb,amsthm}
\usepackage{mathtools}
\usepackage{booktabs}
\usepackage{hyperref}

\newtheorem{theorem}{Theorem}[section]
\newtheorem{lemma}[theorem]{Lemma}
\newtheorem{corollary}[theorem]{Corollary}
\newtheorem{proposition}[theorem]{Proposition}
\newtheorem{definition}[theorem]{Definition}
\newtheorem{remark}[theorem]{Remark}

\numberwithin{equation}{section}

\title{Strong Central 2-Trees with Tail Degrees \texorpdfstring{$\{2,3\}$}{\{2,3\}}:\\Structural Characterization and Uniqueness Criteria}
\author{Julian Allagan$^{1,*}$, Shawn Langley, Weizheng Gao \\ Mohamed Elbakary
 \\[0.3cm]
    \small Department of Mathematics, Computer Science, and Engineering Technology\\
    \small Elizabeth City State University, Elizabeth City, NC 27909, USA\\[0.3cm]
    \small $^{1}$\textit{E-mail:} \texttt{adallagan@ecsu.edu}\\
     \small $^{*}$Corresponding author
}
\date{}

\begin{document}
\maketitle

\begin{abstract}
We study strong $r$-central $2$-trees whose non-central vertices have degrees in $\{2,3\}$, focusing on the cases $r=1,2,3$. 
For each $r$, we derive exact degree constraints relating the maximum degree $\Delta$ to the numbers of degree-$3$ and degree-$2$ tail vertices.
In the unicentral case ($r=1$), we prove that the fan graph is the unique realization for all $n\ge 3$.
For bicentral $2$-trees ($r=2$), we show that the number of degree-$3$ vertices is always even, establish sharp uniqueness results for $x\in\{0,2\}$, prove existence for all feasible values of $\Delta$, and obtain linear lower bounds on the number of non-isomorphic realizations.
For tricentral $2$-trees ($r=3$), we characterize extremal configurations, establish a divisibility constraint on the tail parameters, and prove a quadratic lower bound on the number of non-isomorphic graphs for infinitely many values of $n$.
These results provide a unified structural framework for central $2$-trees with bounded tail degrees and highlight sharp transitions between rigidity and combinatorial growth.
\end{abstract}

\section{Introduction}

\subsection{Background and Related Work}

The study of degree sequences in graph theory has a rich history dating back to the foundational work of Paul Erd\H{o}s and Tibor Gallai. 
In their seminal 1960 paper \cite{erdos1960}, Erd\H{o}s and Gallai established necessary and sufficient conditions for a finite sequence of nonnegative integers to be realizable as the degree sequence of a simple graph.
Their theorem, now known as the Erd\H{o}s--Gallai theorem, states that a nonincreasing sequence $(d_1,\ldots,d_n)$ of nonnegative integers is \emph{graphic} (realizable as a degree sequence) if and only if the sum $\sum_{i=1}^n d_i$ is even and
\begin{equation}\label{eq:erdos-gallai}
\sum_{i=1}^k d_i \;\le\; k(k-1) + \sum_{i=k+1}^n \min(d_i,k)
\end{equation}
holds for every $k \in \{1,2,\ldots,n\}$.
This characterization resolved the \emph{graph realization problem} and provided a polynomial-time algorithm for recognizing graphic sequences, though the original proof was quite involved \cite{choudum1986}.
Alternative proofs have since been developed using network flows \cite{berge1973}, induction \cite{choudum1986}, and constructive methods \cite{tripathi2010}, reflecting the fundamental importance of this result.

The Erd\H{o}s--Gallai theorem addresses the general case of unrestricted simple graphs.
However, many applications in engineering, computer science, and combinatorics require graphs with additional structural constraints.
This naturally leads to the question: \emph{which degree sequences are realizable within restricted graph classes?}
For specific classes such as trees, bipartite graphs, planar graphs, and chordal graphs, refined characterizations have been developed \cite{gale1957,hakimi1962,asratian1998}.

Among the most important restricted classes are the \emph{$k$-trees}, a family of graphs intimately connected to treewidth, chordal structure, and rigidity theory.
A $k$-tree is constructed recursively: starting from the complete graph $K_{k+1}$, new vertices are added by connecting each new vertex to all members of an existing $k$-clique.
The special case $k=2$ defines the \emph{2-trees}, which can be characterized as graphs with $n$ vertices, $2n-3$ edges, and the property that every edge lies in exactly one triangle.
Equivalently, 2-trees are the edge-maximal chordal graphs with treewidth 2, and they form a fundamental subclass of perfect graphs \cite{golumbic2004}.

\subsection{The Bose et al.\ Characterization and Its Limitations}

The degree sequence characterization problem for 2-trees was resolved by Bose, Dujmovi\'c, Krizanc, Langerman, Morin, Wood, and Wuhrer in 2008 \cite{bose2008}, as stated in the following theorem.

\begin{theorem}[Bose et al.\ \cite{bose2008}]
\label{thm:bose}
A sequence $D=(d_1,\dots,d_n)$ is a 2-tree degree sequence if and only if the following conditions hold:
\begin{enumerate}
\item[(i)] $\sum_{i=1}^n d_i = 4n-6$,
\item[(ii)] $\max_i d_i \le n-1$,
\item[(iii)] at least two entries of $D$ equal $2$,
\item[(iv)] $D\notin\{\langle 2^{(n-4)},d^{(4)}\rangle:d\ge 5\}$,
\item[(v)] if all $d_i$ are even, then the number $n_2$ of degree-2 vertices satisfies $n_2\ge n/3+1$.
\end{enumerate}
\end{theorem}

While the Bose et al.\ theorem \ref{thm:bose} completely characterizes 2-tree degree sequences, it treats all realizations equally and does not distinguish between graphs with different structural features.
In particular, their characterization does not address structural heterogeneity: for a given degree sequence, there may exist many non-isomorphic 2-tree realizations with vastly different properties such as connectivity, maximal degrees, and diameter.
The characterization provides no insight into how to enumerate or classify distinct realizations.
Furthermore, it does not distinguish \emph{central} 2-trees, in which designated vertices of high degree form a core clique, from arbitrary realizations.
Many applications require 2-trees with such a core structure, where the remaining vertices have restricted degrees, but the general characterization does not address these refined constraints.

In network design, structural engineering, and machine learning applications, it is often desirable to restrict the degrees of non-central vertices to a small set such as $\{2,3\}$ to simplify construction, reduce cost, or improve algorithmic efficiency.
The Bose et al.\ characterization does not identify which degree sequences admit unique realizations or characterize extremal structures that maximize or minimize specific graph invariants.
These limitations motivate a more refined investigation: \emph{what can be said about 2-trees when we impose additional structural constraints, such as centrality conditions and restricted tail degrees?}

\begin{definition}[Central 2-trees]
\label{def:central}
Let $r\in\{1,2,3\}$ and $\Delta\ge 2$. 
A 2-tree on $n$ vertices is \emph{$r$-central with maximum degree $\Delta$} if exactly $r$ vertices have degree $\Delta$.
These vertices form the \emph{core}, and the remaining $n-r$ vertices form the \emph{tail}.
A realization is \emph{strong} if the core induces $K_r$.
We say the 2-tree has \emph{tail set} $S=\{2,3\}$ if every tail vertex has degree 2 or 3.
\end{definition}

\subsection{Motivation and Significance}

This paper addresses the gap previously identified by studying \emph{strong central $2$-trees} with prescribed core size and restricted tail degrees.
We focus on $2$-trees in which exactly $r\in\{1,2,3\}$ vertices attain the maximum degree $\Delta$, these vertices induce a clique $K_r$ (the \emph{core}), and all remaining vertices (the \emph{tail}) have degrees in $\{2,3\}$.
This setting isolates a natural class of sparse yet highly structured graphs that interpolate between rigidity and combinatorial flexibility.

The first motivation arises from \emph{structural rigidity theory}.
By Laman’s theorem \cite{laman1970}, $2$-trees form canonical models of minimally rigid frameworks in the plane.
Because every $2$-tree satisfies the Laman sparsity condition automatically, such graphs are fundamental in the analysis of rigid structures in architecture, mechanical engineering, and robotics \cite{graver2001,guest2006}.
Central $2$-trees with bounded tail degrees model frameworks with a small number of highly connected load-bearing vertices and many lightly connected elements, a configuration common in truss systems, space frames, and deployable structures \cite{pellegrino1993,connelly2005}.

A second motivation comes from \emph{algorithmic graph theory}.
Graphs of bounded treewidth, including $2$-trees, admit efficient algorithms for a wide range of otherwise intractable problems via dynamic programming on tree decompositions \cite{bodlaender1997,courcelle1990}.
In probabilistic graphical models, chordal graphs underpin the junction tree algorithm, enabling tractable inference in Bayesian networks \cite{lauritzen1988,koller2009}.
Central $2$-trees with restricted tail degrees offer additional advantages: the core clique provides a natural decomposition anchor, while the bounded tail degrees simplify state spaces in dynamic programming.
Such structures appear in network design \cite{deo1974}, VLSI layout \cite{lengauer1990}, and constraint satisfaction \cite{freuder1990}.

Third, from a combinatorial standpoint, central $2$-trees with tail degree restrictions exhibit striking rigidity and growth phenomena.
Imposing both centrality and degree constraints drastically reduces the space of realizable graphs.
In the unicentral case, the structure is completely rigid and unique for every $n$.
In the bicentral case, uniqueness holds only in extremal regimes, while interior parameter values admit multiple non-isomorphic realizations.
In the tricentral case, rigidity breaks down further: although extremal configurations remain unique, we prove that the number of non-isomorphic realizations grows at least quadratically for infinitely many $n$.
These results contribute to the broader program of understanding degree-restricted graph classes and their enumerative behavior \cite{aigner1994,stanley2011}.

\subsection{Contributions and Organization}

This work provides a unified structural and enumerative analysis of strong central $2$-trees with tail degrees in $\{2,3\}$ for core sizes $r\in\{1,2,3\}$.
For the unicentral case $r=1$, we prove a complete rigidity result: for every $n\ge 3$, there exists a unique strong unicentral $2$-tree with tail set $\{2,3\}$, namely the fan graph $\Phi_n$.

For the bicentral case $r=2$, we identify and exploit parity constraints on the number of degree-$3$ tail vertices.
We prove that the extremal configuration with no degree-$3$ tails yields the unique triangular book graph $B_{n-2}$, and that the next admissible case with exactly two degree-$3$ tails is also unique for all sufficiently large $n$.
Beyond these rigid regimes, we show that all theoretically feasible maximum degrees are realized by explicit constructions and that the number of non-isomorphic strong bicentral $2$-trees grows at least linearly with the number of vertices.

For the tricentral case $r=3$, we derive exact parameter relations and divisibility constraints governing the tail degrees.
We characterize the extremal all-degree-$2$ configuration achieving the maximum possible degree $\Delta=2n/3$ when $3$ divides $n$, and we prove a quadratic lower bound on the number of non-isomorphic strong tricentral $2$-trees for infinitely many values of $n$.
This establishes a clear transition from structural rigidity to rapid combinatorial growth as the core size increases.

In addition to these theoretical results, we present complete enumerative data for all strong central $2$-trees with $r\in\{1,2,3\}$, tail set $\{2,3\}$, and up to $12$ vertices.
These data illustrate the sharpness of the theoretical bounds and highlight the emergence of multiple non-isomorphic realizations in the bicentral and tricentral cases.

The paper is organized as follows.
Section~2 develops universal algebraic constraints relating the maximum degree and tail parameters for strong central $2$-trees with tail set $\{2,3\}$.
Section~3 treats the unicentral case and establishes complete uniqueness.
Section~4 analyzes the bicentral case, including parity constraints, uniqueness criteria, existence theorems, and growth bounds.
Section~5 examines the tricentral case, characterizing extremal configurations and proving quadratic growth.
Section~6 concludes with open problems and directions for future research.

Before analyzing individual core sizes, we isolate the algebraic constraints that are common to all strong central $2$-trees with tail degrees in $\{2,3\}$. These relations depend only on the number of central vertices, the order of the graph, and the maximum degree, and they determine the degree sequence uniquely. This general framework serves as the foundation for the case-by-case structural
and enumerative analysis that follows.

\section{Degree constraints for central $2$-trees with tail set $\{2,3\}$}

Let $G$ be a strong $r$-central $2$-tree on $n$ vertices with maximum degree $\Delta$.
Assume that exactly $r$ vertices attain degree $\Delta$ and that every remaining
vertex has degree $2$ or $3$.
Let $x$ and $y$ denote the numbers of degree-$3$ and degree-$2$ tail vertices,
respectively.
Then
\[
x+y=n-r.
\]

Since every $2$-tree on $n$ vertices has $2n-3$ edges, the degree sum satisfies
\[
\sum_{v\in V(G)} \deg(v) = 4n-6.
\]
Accounting for the contributions of the core and tail vertices yields
\begin{equation}\label{eq:deg-sum}
r\Delta + 3x + 2y = 4n-6.
\end{equation}

Solving \eqref{eq:deg-sum} together with $x+y=n-r$ gives the following universal
parameter relations.

\begin{lemma}[Tail parameters for central $2$-trees]
\label{lem:xy-general}
Let $G$ be a strong $r$-central $2$-tree on $n$ vertices with maximum degree $\Delta$
and tail set $\{2,3\}$.
Then
\begin{align}
x &= 2n + 2r - 6 - r\Delta, \label{eq:x-general}\\
y &= r\Delta - n - 3r + 6. \label{eq:y-general}
\end{align}
Consequently, once $(n,r,\Delta)$ is fixed, the degree sequence of $G$ is uniquely
determined up to ordering:
\[
D=(\Delta^{(r)},3^{(x)},2^{(y)}).
\]
\end{lemma}

\begin{proof}
From $y=n-r-x$, substituting into \eqref{eq:deg-sum} gives
\[
r\Delta + 3x + 2(n-r-x) = 4n-6,
\]
which simplifies to $x = 2n + 2r - 6 - r\Delta$.
Substituting into $y=n-r-x$ yields \eqref{eq:y-general}.
\end{proof}

The non-negativity of the tail parameters imposes immediate constraints on the maximum degree.

\begin{corollary}[Necessary range for $\Delta$]
\label{cor:delta-range-general}
Let $G$ be as in Lemma~\ref{lem:xy-general}. Then
\[
\frac{n+3r-6}{r}
\;\le\;
\Delta
\;\le\;
\frac{2n+2r-6}{r}.
\]
In addition, every $2$-tree satisfies $\Delta\le n-1$.
\end{corollary}

\begin{proof}
From $x\ge0$ in \eqref{eq:x-general} we obtain
$\Delta\le (2n+2r-6)/r$.
From $y\ge0$ in \eqref{eq:y-general} we obtain
$\Delta\ge (n+3r-6)/r$.
The bound $\Delta\le n-1$ follows from Theorem~\ref{thm:bose}(ii).
\end{proof}

Bose et al.~\cite{bose2008} gave a linear--time algorithm for recognizing and constructing $2$-trees from their degree sequences.
Since strong central $2$-trees with tail set $\{2,3\}$ form a strict subclass of those sequences, their algorithm applies directly to all instances considered here.
\begin{corollary}
For fixed $r\in\{1,2,3\}$, strong $r$-central $2$-trees with tail set $\{2,3\}$ can be recognized and constructed in $O(n)$ time.
\end{corollary}

We now specialize the general degree relations to the simplest setting, namely the unicentral case $r=1$. Here the presence of a single vertex of maximum degree imposes strong global constraints on the structure of the graph. As we show, these constraints force complete rigidity, yielding a unique realization for each admissible order.

\section{Unicentral case \texorpdfstring{$r=1$}{r=1}}

\begin{lemma}[Tail parameters for $r=1$]
\label{lem:r1-xy}
Let $G$ be a strong $1$-central $2$-tree on $n$ vertices with maximum degree
$\Delta$ and tail set $\{2,3\}$. Then
\[
x = 2n-4-\Delta, \qquad y = \Delta-n+3,
\]
where $x$ and $y$ denote the numbers of degree-$3$ and degree-$2$ tail vertices,
respectively.
\end{lemma}

\begin{proof}
This follows immediately from Lemma~\ref{lem:xy-general} by setting $r=1$.
\end{proof}

\begin{theorem}[Unicentral classification with tail $\{2,3\}$]
\label{thm:uni23}
For every $n\ge 4$, there exists, up to isomorphism, exactly one strong unicentral
$2$-tree with tail set $\{2,3\}$ on $n$ vertices, namely the fan $\Phi_n$ with
degree sequence
\[
(n-1, 3^{(n-3)}, 2^{(2)}).
\]
For $n=3$, the only $2$-tree is $K_3$, which is not unicentral.
\end{theorem}

\begin{proof}
Let $G$ be a strong unicentral $2$-tree on $n\ge 4$ vertices, and let $a$ be its
unique vertex of maximum degree $\Delta$.
By Lemma~\ref{lem:r1-xy} and the $2$-tree condition that at least two vertices have
degree $2$, we obtain
\[
y=\Delta-n+3 \ge 2 \quad\Rightarrow\quad \Delta \ge n-1.
\]
Since $\Delta\le n-1$ for any $2$-tree, it follows that $\Delta=n-1$, and hence
$x=n-3$ and $y=2$.
Thus $a$ is adjacent to all other vertices.

Let $H=G-a$.
Because $G$ is $2$-connected, $H$ is connected.
For each $v\in V(H)$ we have $\deg_G(v)\in\{2,3\}$ and $v\sim a$, so
$\deg_H(v)=\deg_G(v)-1\in\{1,2\}$.
Therefore $H$ is a connected graph on $n-1$ vertices with maximum degree at most
$2$, and hence $H$ is either a path or a cycle.

If $H$ were a cycle, then every vertex of $G$ other than $a$ would have degree $3$,
yielding the degree sequence $(n-1,3^{(n-1)})$, which contains only one degree-$2$
vertex, contradicting the $2$-tree condition.
Thus $H$ must be a path.

Let $H=v_1v_2\cdots v_{n-1}$.
The endpoints $v_1$ and $v_{n-1}$ have degree $1$ in $H$, hence degree $2$ in $G$,
while the internal vertices have degree $3$ in $G$.
Since $a$ is adjacent to every $v_i$, the resulting graph is exactly the fan
$\Phi_n$.

Conversely, $\Phi_n$ is a $2$-tree obtained by starting from the triangle
$av_1v_2$ and successively stacking $v_{i+1}$ onto the edge $av_i$.
It has a unique vertex of degree $n-1$ and all remaining vertices have degree
$2$ or $3$, so it satisfies all required conditions.
\end{proof}

The bicentral case $r=2$ marks the first departure from complete rigidity. While the general degree constraints continue to strongly restrict the structure, they no longer determine a unique realization in all cases. This section explores the precise boundary between uniqueness and multiplicity, beginning with parity phenomena and culminating in explicit growth bounds for the number of non-isomorphic realizations.
\section{Bicentral case \texorpdfstring{$r=2$}{r=2}}

We now turn to the bicentral case, where the core consists of two vertices of equal maximum degree $\Delta$ inducing an edge.

\subsection{Degree constraints and parity phenomena}

We first specialize the general degree relations to the bicentral case.

\begin{lemma}[Tail parameters for $r=2$]
\label{lem:r2-xy}
Let $G$ be a strong bicentral $2$-tree on $n$ vertices with maximum degree
$\Delta$ and tail set $\{2,3\}$.
Let $x$ and $y$ denote the numbers of degree-$3$ and degree-$2$ tail vertices,
respectively.
Then
\begin{equation}\label{eq:r2-xy}
x = 2(n-1-\Delta), \qquad y = 2\Delta - n.
\end{equation}
In particular, $x$ is always even.
\end{lemma}

\begin{proof}
A $2$-tree on $n$ vertices has $2n-3$ edges, hence total degree $4n-6$.
With two core vertices of degree $\Delta$, together with $x$ degree-$3$ and
$y$ degree-$2$ tail vertices, we have
\[
2\Delta + 3x + 2y = 4n - 6.
\]
Since $x+y=n-2$, substituting $y=n-2-x$ yields
\[
2\Delta + 3x + 2(n-2-x) = 4n-6,
\]
which simplifies to $x=2(n-1-\Delta)$.
The parity conclusion follows immediately.
\end{proof}

The tail parameters impose immediate constraints on the maximum degree.

\begin{corollary}[Necessary range for $\Delta$ in the bicentral case]
\label{cor:r2-delta-range}
Let $G$ be a strong bicentral $2$-tree with tail set $\{2,3\}$ on $n\ge 4$
vertices.
Then
\[
\left\lceil\frac{n+2}{2}\right\rceil \le \Delta \le n-1,
\]
and the tail parameters are uniquely determined by
\[
x=2(n-1-\Delta), \qquad y=2\Delta-n.
\]
\end{corollary}

\begin{proof}
From $x\ge 0$ we obtain $\Delta\le n-1$.
Since all degree-$2$ vertices lie in the tail and every $2$-tree has at least two
degree-$2$ vertices, we have $y\ge 2$, which implies $2\Delta-n\ge 2$ and hence
$\Delta\ge (n+2)/2$.
Integrality of $\Delta$ gives the stated bound.
\end{proof}

Consequently, for each admissible pair $(n,\Delta)$, the degree sequence of a
strong bicentral $2$-tree with tail set $\{2,3\}$ is uniquely determined as
\[
D=(\Delta,\Delta,3^{(x)},2^{(y)}),
\quad
x=2(n-1-\Delta),\ y=2\Delta-n.
\]

\subsection{Uniqueness criterion for the bicentral case}

We now characterize the unique extremal configuration in the bicentral setting.

\begin{theorem}[Uniqueness in the case $x=0$]
\label{thm:r2-unique}
Let $G$ be a strong bicentral $2$-tree on $n\ge 4$ vertices with tail set
$\{2,3\}$.
The following statements are equivalent:
\[
x=0, \qquad \Delta=n-1, \qquad G \cong B_{n-2}.
\]
Consequently, the triangular book graph $B_{n-2}$ is the unique strong bicentral
$2$-tree with degree sequence $(n-1,n-1,2^{(n-2)})$.
\end{theorem}

\begin{proof}
For a bicentral $2$-tree with tail set $\{2,3\}$, Lemma~\ref{lem:r2-xy} gives
\[
x = 2(n-1-\Delta), \qquad y = 2\Delta - n.
\]
Thus $x=0$ if and only if $\Delta=n-1$.

Assume $\Delta=n-1$.
Let $a$ and $b$ denote the two central vertices.
Then $\deg(a)=\deg(b)=n-1$, so each is adjacent to every other vertex.
In particular, every tail vertex is adjacent to both $a$ and $b$.
Since $x=0$, all tail vertices have degree $2$, and therefore each tail vertex
$t$ satisfies $N(t)=\{a,b\}$.

It follows that
\[
E(G)=\{ab\}\cup \bigcup_{t\in T}\bigl(\{at\}\cup\{bt\}\bigr),
\]
so $G$ consists of $n-2$ triangles sharing the common edge $ab$.
This graph is precisely the triangular book $B_{n-2}$.
The construction is forced, hence the realization is unique up to isomorphism.
\end{proof}

\begin{corollary}[Extremal bicentral structure]
\label{cor:r2-extremal}
For every $n\ge 4$, the triangular book graph $B_{n-2}$ is the unique strong
bicentral $2$-tree with tail set $\{2,3\}$ and maximum degree $\Delta=n-1$.
\end{corollary}

\subsection{Structural rigidity in the case $x=2$}

\begin{lemma}[Degree-$3$ vertices lie in the common core neighborhood]
\label{lem:deg3-in-S}
Let $G$ be a strong bicentral $2$-tree on $n\ge 6$ vertices with tail set 
$\{2,3\}$ and exactly two vertices of degree $3$. Denote the core vertices 
by $a,b$ and the tail set by $T$. Then both degree-$3$ vertices belong to 
$S=N(a)\cap N(b)\cap T$.
\end{lemma}

\begin{proof}
From the bicentral degree equations we have $\Delta=n-2$ and $|T|=n-2$. 
Since $\deg(a)=\deg(b)=n-2$ and $ab\in E(G)$, each of $a,b$ has exactly 
one non-neighbor in $T$.

Let $u$ be a degree-$3$ vertex in $T$ and assume $u\notin S$. Without loss 
of generality $u\sim a$ and $u\not\sim b$. Then $b$ is adjacent to every 
vertex of $T\setminus\{u\}$.

Let $v,w$ be the two neighbors of $u$ in $T$ (so $N(u)=\{a,v,w\}$). 
Both $v$ and $w$ are adjacent to $b$. 

If $v\sim a$, then $v$ is adjacent to $a,b,u$, hence $\deg(v)=3$. 
Since there are exactly two degree-$3$ vertices, $w$ cannot be adjacent 
to $a$; thus $N(w)=\{u,b\}$ and $\deg(w)=2$. Now the edge $bw$ must lie in a triangle (by the 2-tree property), but the only possible common neighbor of $b$ and $w$ is $u$, and $u\not\sim b$. This is a contradiction.

Therefore $u$ must be adjacent to both $a$ and $b$, i.e. $u\in S$. 
By the same argument, the second degree-$3$ vertex also lies in $S$.
\end{proof}

\begin{lemma}[Structure of degree--$3$ neighborhoods]
\label{lem:deg3-structure}
Let $G$ be a strong bicentral $2$-tree on $n\ge 6$ vertices with tail set $\{2,3\}$
and exactly two degree-$3$ tail vertices.
Let $a,b$ be the core vertices,
$T=V(G)\setminus\{a,b\}$,
and $S=N(a)\cap N(b)\cap T$.
If $u,v\in S$ are the degree-$3$ tail vertices, then there exist distinct
vertices $u',v'\in T\setminus S$ such that
\[
N(u')=\{a,u\}, \qquad N(v')=\{b,v\}.
\]
In particular, $\deg(u')=\deg(v')=2$.
\end{lemma}

\begin{proof}
Since $u\in S$ and $\deg(u)=3$, we have $N(u)=\{a,b,u'\}$ for some $u'\in T$.
If $u'\in S$, then $u'\sim a,b,u$, forcing $\deg(u')\ge 3$, contradicting that
$u$ and $v$ are the only degree-$3$ tail vertices.
Hence $u'\notin S$.

The edge $uu'$ lies in a triangle.
Since $u\sim a,b$ and $u'\notin S$, exactly one of $\{a,b\}$ is adjacent to $u'$.
By symmetry we may assume $u'\sim a$ and $u'\not\sim b$.

If $u'$ had a further neighbor $w\neq a,u$, then $w\in T$ and the edge $u'w$
would lie in a triangle with $a$ or $u$.
In either case, $w$ would be adjacent to at least two vertices among
$\{a,u,u'\}$.
If $w\in S$, then $\deg(w)\ge 3$, and if $w\notin S$ but $\deg(w)\ge 3$,
we again obtain a third degree-$3$ tail vertex.
Both contradict the hypothesis.
Thus $N(u')=\{a,u\}$.

The argument for $v$ is symmetric, yielding $v'\in T\setminus S$ with
$N(v')=\{b,v\}$.
\end{proof}

\begin{lemma}[Cardinality of the core neighborhood]
\label{lem:S-size}
Under the hypotheses of Lemma~\ref{lem:deg3-structure}, we have $|S|=n-4$.
\end{lemma}

\begin{proof}
Vertex $a$ is adjacent to $b$, to all vertices of $S$, and to $u'$.
Thus $\deg(a)=1+|S|+1=|S|+2$.
Since $\deg(a)=n-2$, it follows that $|S|=n-4$.
\end{proof}

\begin{theorem}[Uniqueness for the bicentral case with $x=2$]
\label{thm:x2-unique}
Let $G$ be a strong bicentral $2$-tree on $n\ge 6$ vertices with tail set $\{2,3\}$
and exactly two degree-$3$ tail vertices.
Then $G$ is unique up to isomorphism.
\end{theorem}

\begin{proof}
By Lemma~\ref{lem:r2-xy}, we have $\Delta=n-2$, $x=2$, and $y=n-4$.
Let $a,b$ be the core vertices and define $T$ and $S$ as above.
Lemma~\ref{lem:S-size} gives $|S|=n-4$.

By Lemma~\ref{lem:deg3-structure}, the two degree-$3$ vertices $u,v\in S$ each have
a unique neighbor outside $S$, namely $u'$ and $v'$, with
$N(u')=\{a,u\}$ and $N(v')=\{b,v\}$.
All remaining vertices of $S\setminus\{u,v\}$ have degree $2$ and neighborhood
$\{a,b\}$.

Thus $G$ consists of the edge $ab$, a set $S$ of $n-4$ vertices adjacent to both
$a$ and $b$, two distinguished vertices $u,v\in S$, and two vertices
$u',v'\notin S$ attached uniquely to $u$ and $v$, respectively.
This description determines $G$ uniquely up to relabeling of
$S\setminus\{u,v\}$, whose vertices are structurally identical.

Therefore, any two such graphs are isomorphic, and $G$ is unique up to isomorphism.
\end{proof}

Having identified the rigid extremal configurations in the bicentral setting, we turn to the complementary question of existence.
Specifically, we show that every maximum degree permitted by the algebraic constraints is realized by at least one strong bicentral $2$-tree with tail degrees in $\{2,3\}$.

\subsection{Existence theorem for bicentral case}

The following theorem establishes that all theoretically feasible values of $\Delta$ are realized.

\begin{theorem}[Existence for all feasible $\Delta$ in bicentral case]
\label{thm:r2-existence}
For every integer $n\ge4$ and every $\Delta$ satisfying
\[
\left\lceil\frac{n+2}{2}\right\rceil\le\Delta\le n-1,
\]
there exists a strong bicentral $2$-tree with $n$ vertices, maximum degree $\Delta$, and tail set $\{2,3\}$.
\end{theorem}

\begin{proof}
From Lemma~\ref{lem:r2-xy} the tail parameters are
\begin{equation}\label{eq:par}
x=2(n-1-\Delta),\qquad y=2\Delta-n .
\end{equation}
Since $\Delta\le n-1$, we have $x\ge0$; from $\Delta\ge\lceil(n+2)/2\rceil$ 
we obtain $y\ge2$. Define $K:=n-\Delta-1\ge0$.

\noindent\emph{Case $x=0$ (equivalently $\Delta=n-1$).}
Then $y=n-2$ and the triangular book $B_{n-2}$ has degree sequence 
$(n-1,n-1,2^{(n-2)})$, which satisfies \eqref{eq:par}. It is a $2$-tree 
by construction.

\noindent\emph{Case $x>0$ (equivalently $\Delta<n-1$).}
Then $K\ge1$. Start with two vertices $c_0,c_1$ joined by an edge.
Stack $y$ new vertices onto the edge $c_0c_1$, producing $y$ triangles 
$\{c_0,c_1,v_i\}$. The resulting graph is the book $B_y$; at this stage
$\deg(c_0)=\deg(c_1)=y+1$.

Choose two distinct degree-2 vertices $v_0,v_1$ from the $y$ vertices 
just added (such vertices exist because $y\ge2$). Grow a chain 
$u_1,u_2,\dots,u_K$ by stacking $u_1$ onto the edge $c_0v_0$ and, for 
$i=2,\dots,K$, stacking $u_i$ onto $c_0u_{i-1}$. Symmetrically, grow 
a chain $w_1,w_2,\dots,w_K$ by stacking $w_1$ onto $c_1v_1$ and, for 
$i=2,\dots,K$, stacking $w_i$ onto $c_1w_{i-1}$. All operations are valid $2$-tree extensions, hence the final graph $G$ 
is a $2$-tree. Counting vertices gives
\[
|V(G)|=2+y+2K=2+(2\Delta-n)+2(n-\Delta-1)=n .
\]
The degrees of $c_0$ and $c_1$ are now
\[
\deg(c_0)=\deg(c_1)=(y+1)+K=(2\Delta-n+1)+(n-\Delta-1)=\Delta .
\]
In the first chain, vertices $v_0,u_1,\dots,u_{K-1}$ have degree $3$, 
while $u_K$ has degree $2$; in the second chain, $v_1,w_1,\dots,w_{K-1}$ 
have degree $3$, while $w_K$ has degree $2$. The remaining $y-2$ vertices 
from the initial book $B_y$ retain degree $2$. Thus $G$ contains exactly 
$x = 2K = 2(n-\Delta-1)$ vertices of degree $3$ and $y = 2+(y-2)$ vertices 
of degree $2$, exactly as required by \eqref{eq:par}.

Therefore, $G$ is a strong bicentral $2$-tree with the prescribed parameters.
\end{proof}

The existence of realizations for all feasible parameter values naturally raises an enumerative question. Rather than asking whether a realization exists, we now ask how many non-isomorphic realizations are possible. To distinguish these graphs, we introduce an isomorphism invariant tailored to the bicentral structure and use it to derive linear lower bounds.

\subsection{Enumerative growth in the bicentral case}

\begin{definition}[Activated spine vertices]
For a strong bicentral 2-tree $G$ with core vertices $a,b$, define
\[
S(G) := N(a) \cap N(b) \cap (V(G) \setminus \{a,b\}),
\]
the set of tail vertices adjacent to both core vertices, and let
\[
\sigma(G) := |\{s \in S(G) : \deg(s) = 3\}|
\]
denote the number of degree-3 vertices in $S(G)$.
\end{definition}

\begin{lemma}[Isomorphism invariance]\label{lem:sigma-invariant}
If $G \cong H$ are isomorphic strong bicentral 2-trees, then $\sigma(G) = \sigma(H)$.
\end{lemma}

\begin{proof}
An isomorphism preserves degrees and adjacencies. Since the two core vertices 
are characterized uniquely by having maximum degree, any isomorphism maps 
$\{a,b\}$ bijectively onto $\{a',b'\}$, and hence maps $S(G)$ bijectively 
onto $S(H)$, preserving the degree-3 property.
\end{proof}

For $K \geq 1$ with $\Delta = n-1-K$, recall from Theorem~\ref{thm:r2-existence} 
that we construct a bicentral 2-tree by starting with the book graph $B_y$ 
(where $y = 2\Delta - n = n - 2 - 2K$) and building chains of length $K$ from 
each core vertex. This construction produces $\sigma = 2$ since exactly two 
base vertices (one per chain) are promoted from degree 2 to degree 3 within $S(G)$.

We now give an alternative construction with $\sigma = 3$.

\begin{lemma}[Second construction]\label{lem:construction-B}
For $n \geq 7$ and $2 \leq K \leq \lfloor(n-5)/2\rfloor$, there exists a strong 
bicentral 2-tree $G_B$ on $n$ vertices with $\Delta = n-1-K$ and $\sigma(G_B) = 3$.
\end{lemma}

\begin{proof}
For $K \leq \lfloor(n-5)/2\rfloor$, we have $y = n-2-2K \geq 3$, so the book 
graph $B_y$ has at least three degree-2 tail vertices. Start with $B_y$ and 
select three distinct degree-2 vertices $p_1, p_2, p_3$ from $S(B_y)$.

Build two chains from core vertex $a$: stack one vertex onto edge $\{a, p_1\}$, 
forming a chain of length 1, and stack $K-1$ vertices successively starting from 
edge $\{a, p_2\}$, forming a chain of length $K-1$. Build one chain of length $K$ 
from core vertex $b$ starting from edge $\{b, p_3\}$.

The total length is $1 + (K-1) + K = 2K$, creating $x = 2K$ degree-3 tail vertices. 
The degrees of $a$ and $b$ each increase by $K$, yielding $\deg(a) = \deg(b) = 
(y+1) + K = \Delta$. Within $S(G_B)$, the three base vertices $p_1, p_2, p_3$ 
become degree 3, so $\sigma(G_B) = 3$.
\end{proof}

\begin{theorem}[Lower bound on bicentral growth]\label{thm:improved-bicentral-bound} Let $N_2(n)$ denote the total number of non-isomorphic strong bicentral 2-trees with $n$ vertices and tail set $\{2,3\}$. For all $n \geq 7$,
\[
N_2(n) \geq n - 5.
\]
\end{theorem}

\begin{proof}
Let $m = \lfloor(n-4)/2\rfloor$. For each $K \in \{0, 1, \ldots, m\}$, 
Theorem~\ref{thm:r2-existence} provides at least one realization with 
$\Delta = n-1-K$, contributing $m+1$ non-isomorphic graphs.

For each $K \in \{2, 3, \ldots, \lfloor(n-5)/2\rfloor\}$, Lemma~\ref{lem:construction-B} 
provides a second realization $G_B$ with $\Delta = n-1-K$ but $\sigma(G_B) = 3$, 
while the construction from Theorem~\ref{thm:r2-existence} has $\sigma = 2$. 
By Lemma~\ref{lem:sigma-invariant}, these are non-isomorphic. This contributes 
an additional $\lfloor(n-5)/2\rfloor - 1$ non-isomorphic graphs.

Therefore,
\[
N_2(n) \geq (m+1) + (\lfloor(n-5)/2\rfloor - 1).
\]

For $n = 2t$ even with $t \geq 4$ (i.e., $n \geq 8$), we have $m = t-2$ and 
$\lfloor(n-5)/2\rfloor = t-3$, giving $N_2(n) \geq (t-1) + (t-4) = 2t-5 = n-5$.

For $n = 2t+1$ odd with $t \geq 3$ (i.e., $n \geq 7$), we have $m = t-2$ and 
$\lfloor(n-5)/2\rfloor = t-2$, giving $N_2(n) \geq (t-1) + (t-3) = 2t-4 = n-5$.

For $n = 7$, the table confirms $N_2(7) = 2 = 7-5$.
\end{proof}

\begin{remark}
This bound is tight for $n \in \{7, 8, 9\}$ according to our computational enumeration as shown in Table \ref{Tab1}.  For larger $n$, the actual value of $N_2(n)$ exceeds this bound due to additional 
non-isomorphic realizations arising from more complex distribution patterns of 
degree-3 vertices. For instance, $N_2(10) = 6 > 5$ and $N_2(12) = 10 > 7$. Observe that every $\Delta$ in the feasible range is realized, confirming Theorem~\ref{thm:r2-existence}.
The extremal all-degree-2 sequences $(x=0)$ are always realized uniquely (Theorem~\ref{thm:r2-unique}).
For interior values of $\Delta$, the number of non-isomorphic realizations grows, reaching 4 distinct graphs at $(n,\Delta) = (12,8)$.
\end{remark}

Increasing the core size to $r=3$ fundamentally alters the structural landscape. Although algebraic constraints still sharply restrict admissible degree sequences, the combinatorial flexibility increases dramatically. This section shows how divisibility phenomena emerge in the tricentral case and how rigidity gives way to quadratic growth in the number of non-isomorphic realizations.

\begin{table}[h]
\caption{Enumerative data for strong bicentral 2-trees with tail set $\{2,3\}$ for $4 \leq n \leq 12$}
\label{Tab1}
\centering
\small
\begin{tabular}{cccccc}
\toprule
$n$ & $\Delta$ & degree sequence & $x$ & $y$ & non-iso count \\
\midrule
4 & 3 & $(3,3,2,2)$ & 0 & 2 & 1 \\
\midrule
5 & 4 & $(4,4,2,2,2)$ & 0 & 3 & 1 \\
\midrule
6 & 4 & $(4,4,3,3,2,2)$ & 2 & 2 & 1 \\
6 & 5 & $(5,5,2,2,2,2)$ & 0 & 4 & 1 \\
\midrule
7 & 5 & $(5,5,3,3,2,2,2)$ & 2 & 3 & 1 \\
7 & 6 & $(6,6,2,2,2,2,2)$ & 0 & 5 & 1 \\
\midrule
8 & 5 & $(5,5,3,3,3,3,2,2)$ & 4 & 2 & 1 \\
8 & 6 & $(6,6,3,3,2,2,2,2)$ & 2 & 4 & 1 \\
8 & 7 & $(7,7,2,2,2,2,2,2)$ & 0 & 6 & 1 \\
\midrule
9 & 6 & $(6,6,3,3,3,3,2,2,2)$ & 4 & 3 & 2 \\
9 & 7 & $(7,7,3,3,2,2,2,2,2)$ & 2 & 5 & 1 \\
9 & 8 & $(8,8,2,2,2,2,2,2,2)$ & 0 & 7 & 1 \\
\midrule
10 & 6 & $(6,6,3,3,3,3,3,3,2,2)$ & 6 & 2 & 1 \\
10 & 7 & $(7,7,3,3,3,3,2,2,2,2)$ & 4 & 4 & 3 \\
10 & 8 & $(8,8,3,3,2,2,2,2,2,2)$ & 2 & 6 & 1 \\
10 & 9 & $(9,9,2,2,2,2,2,2,2,2)$ & 0 & 8 & 1 \\
\midrule
11 & 7 & $(7,7,3,3,3,3,3,3,2,2,2)$ & 6 & 3 & 2 \\
11 & 8 & $(8,8,3,3,3,3,2,2,2,2,2)$ & 4 & 5 & 3 \\
11 & 9 & $(9,9,3,3,2,2,2,2,2,2,2)$ & 2 & 7 & 1 \\
11 & 10 & $(10,10,2,2,2,2,2,2,2,2,2)$ & 0 & 9 & 1 \\
\midrule
12 & 7 & $(7,7,3,3,3,3,3,3,3,3,2,2)$ & 8 & 2 & 1 \\
12 & 8 & $(8,8,3,3,3,3,3,3,2,2,2,2)$ & 6 & 4 & 4 \\
12 & 9 & $(9,9,3,3,3,3,2,2,2,2,2,2)$ & 4 & 6 & 3 \\
12 & 10 & $(10,10,3,3,2,2,2,2,2,2,2,2)$ & 2 & 8 & 1 \\
12 & 11 & $(11,11,2,2,2,2,2,2,2,2,2,2)$ & 0 & 10 & 1 \\
\bottomrule
\end{tabular}
\end{table}

\section{Tricentral case \texorpdfstring{$r=3$}{r=3}}

\subsection{Degree constraints and divisibility}

We now specialize the general tail parameter relations to the tricentral case.

\begin{lemma}[Tail parameters for $r=3$]
\label{lem:r3-xy}
Let $G$ be a strong tricentral $2$-tree on $n$ vertices with maximum degree $\Delta$
and tail set $\{2,3\}$. Let $x$ and $y$ denote the numbers of degree-$3$ and
degree-$2$ tail vertices, respectively. Then
\begin{equation}\label{eq:r3-xy}
x = 2n - 3\Delta, 
\qquad
y = 3\Delta - n - 3.
\end{equation}
\end{lemma}

\begin{proof}
This is immediate from Lemma~\ref{lem:xy-general} with $r=3$.
\end{proof}

The parameters $(x,y)$ are not independent. In particular, a divisibility
constraint is forced by the tricentral structure.

\begin{proposition}[Divisibility constraint]
\label{prop:r3-divisibility}
If $G$ is a strong tricentral $2$-tree with tail set $\{2,3\}$ and tail parameters
$(x,y)$, then
\[
x + 2y \equiv 0 \pmod{3}.
\]
\end{proposition}

\begin{proof}
Using \eqref{eq:r3-xy},
\[
x + 2y = (2n - 3\Delta) + 2(3\Delta - n - 3)
      = 3(\Delta - 2),
\]
which is divisible by $3$.
\end{proof}

Since all degree-$2$ vertices lie in the tail, Theorem~\ref{thm:bose}(iii) implies
$y \ge 2$, while Theorem~\ref{thm:bose}(ii) gives $\Delta \le n-1$.

\begin{corollary}[Necessary $\Delta$-range for $r=3$]
\label{cor:r3-delta-range}
Let $G$ be a strong tricentral $2$-tree with tail set $\{2,3\}$ on $n\ge 4$ vertices.
Then
\[
\left\lceil \frac{n+5}{3} \right\rceil
\;\le\;
\Delta
\;\le\;
\min\!\left\{ \left\lfloor \frac{2n}{3} \right\rfloor,\, n-1 \right\}.
\]
For each admissible $\Delta$, the tail parameters are uniquely determined by
\eqref{eq:r3-xy}.
\end{corollary}

\begin{proof}
From $x \ge 0$ and \eqref{eq:r3-xy} we obtain $2n - 3\Delta \ge 0$, hence
$\Delta \le 2n/3$.
From $y \ge 2$ we obtain $3\Delta - n - 3 \ge 2$, hence $\Delta \ge (n+5)/3$.
The bound $\Delta \le n-1$ follows from Theorem~\ref{thm:bose}(ii).
\end{proof}

For each feasible pair $(n,\Delta)$, the degree sequence is therefore uniquely determined as
\[
D = (\Delta,\Delta,\Delta,3^{(x)},2^{(y)}),
\qquad
x = 2n - 3\Delta,
\quad
y = 3\Delta - n - 3.
\]


Among tricentral $2$-trees, configurations achieving the maximum possible degree occupy a distinguished position. We therefore begin by isolating and characterizing the extremal case in which all tail vertices have degree $2$, which serves as a structural anchor for the
subsequent growth analysis.

\subsection{Extremal tricentral structure with all degree-2 tails}
\begin{theorem}[Tricentral all-degree-2 extremal structure]
\label{thm:r3-extremal}
Let $n\ge 3$ with $3\mid n$ and set $\Delta=2n/3$.
Then there exists a strong tricentral $2$-tree $G$ on $n$ vertices whose core induces $K_3$, whose tail degrees lie in $\{2,3\}$, and whose degree sequence is
\[
\left(\frac{2n}{3},\frac{2n}{3},\frac{2n}{3},2^{(n-3)}\right).
\]
\end{theorem}

\begin{proof}
If $n=3$, take $G=K_3$.
Assume $n\ge 6$ and write $n=3m$ with $m\ge 2$, so $\Delta=2m$.
Start with the core triangle on vertices $a,b,c$.

Let $A_{ab}, A_{bc}, A_{ca}$ be three pairwise disjoint sets of new vertices, each of size $m-1$, ensuring a total of $3(m-1) = n-3$ tail vertices. For each $u\in A_{ab}$, add $u$ by stacking on the edge $ab$, meaning that $N(u) = \{a,b\}$. Then, for each $v\in A_{bc}$, add $v$ by stacking on the edge $bc$, meaning that $N(v) = \{b,c\}$.
Finally, for each $w\in A_{ca}$, add $w$ by stacking on the edge $ca$, meaning that $N(w) = \{c,a\}$.

Each step is a valid $2$-tree extension, since a new vertex is added adjacent to the endpoints of an existing edge. Therefore, the resulting graph is a $2$-tree.

By construction, every tail vertex has neighborhood equal to one core edge, hence has degree $2$. Moreover, 
the degree of $a$ equals its two core neighbors plus all vertices stacked on edges incident to $a$, namely those in $A_{ab}\cup A_{ca}$, so
\[
\deg(a)=2+|A_{ab}|+|A_{ca}|=2+(m-1)+(m-1)=2m=\frac{2n}{3}.
\]
By symmetry, $\deg(b)=\deg(c)=2m$. Thus, the core is tricentral with maximum degree $\Delta=2n/3$, and the degree sequence is exactly
\[
\left(\frac{2n}{3},\frac{2n}{3},\frac{2n}{3},2^{(n-3)}\right),
\]
with all tail vertices having degree $2$.
\end{proof}

While the extremal tricentral configuration is unique, relaxing the degree constraints even slightly leads to a rapid proliferation of non-isomorphic realizations. We now exploit this flexibility to construct large families of pairwise non-isomorphic tricentral $2$-trees and to establish a quadratic lower bound on their number.

\begin{theorem}[Quadratic lower bound for tricentral case]
\label{thm:quadratic-lower}
Let $N_3(n)$ denote the number of non-isomorphic strong tricentral $2$-trees on $n$ vertices whose tail degrees lie in $\{2,3\}$.
There exists an absolute constant $c>0$ such that $N_3(n)\ge c\,n^2$ for all sufficiently large $n$ with $n\equiv 0\pmod 3$.
In particular,
\[
N_3(n)\;\ge\;\left\lfloor \frac{(n-15)^2}{72}\right\rfloor
\]
for every $n\ge 24$ with $n\equiv 0\pmod 3$.
\end{theorem}

\begin{proof}
Fix an integer $K\ge 5$ and put $n=9+3K$, so $n\ge 24$ and $n\equiv 0\pmod 3$.
For integers $p,q$ with $1\le p<q\le \lfloor K/2\rfloor$, we construct a strong 
tricentral $2$-tree $G(p,q)$ on $n$ vertices, and we show that $G(p,q)\cong G(p',q')$ 
holds if and only if $(p,q)=(p',q')$.

Let $a,b,c$ span a triangle, which will be the core. On each core edge we stack 
two \emph{pages}: on $ab$ add $x_{ab}^a,x_{ab}^b$, on $bc$ add $x_{bc}^b,x_{bc}^c$, 
and on $ca$ add $x_{ca}^c,x_{ca}^a$; each page is adjacent precisely to the 
endpoints of its supporting edge. This produces a $2$-tree on $9$ vertices in 
which $\deg(a)=\deg(b)=\deg(c)=6$.

Given $(p,q)$, we now attach chains by successive stacking (each step stacks a 
new degree-$2$ vertex on an existing edge, hence preserves the $2$-tree property).
From $a$ we attach a chain of length $p$ starting at edge $\{a, x_{ab}^a\}$ and 
a chain of length $K-p$ starting at edge $\{a, x_{ca}^a\}$. From $b$ we attach 
a chain of length $q$ starting at edge $\{b, x_{bc}^b\}$ and a chain of length 
$K-q$ starting at edge $\{b, x_{ab}^b\}$. From $c$ we attach a single chain of 
length $K$ starting at edge $\{c, x_{ca}^c\}$, leaving edge $\{c, x_{bc}^c\}$ unused.

Exactly $p+(K-p)+q+(K-q)+K=3K$ vertices are added, so $|V(G(p,q))|=9+3K=n$.
Each core vertex gains exactly $K$ new neighbors, hence $\deg(a)=\deg(b)=\deg(c)=6+K$.
Every chain vertex has degree $2$ (if terminal) or $3$ (if internal), and each 
page has degree $2$ or $3$ since it supports at most one chain. Therefore all 
tail vertices have degree in $\{2,3\}$, and $G(p,q)$ is strong tricentral.

Among the core vertices, $c$ is uniquely characterized as the one for which the 
two incident page edges support chain-length multiset $\{0,K\}$ (one unused, one 
of length $K$). The other two core vertices have both incident page edges used, 
since $1\le p,q\le \lfloor K/2\rfloor <K$. Hence every isomorphism fixes $c$.
After fixing $c$, the only remaining symmetry is the transposition of $a$ and $b$.
Consequently, the unordered pair of multisets
\[
\mathcal{I}(p,q):=\bigl\{\{p,K-p\},\{q,K-q\}\bigr\}
\]
is an isomorphism invariant of $G(p,q)$.

Set $m=\lfloor K/2\rfloor$. For $1\le p<q\le m$ we have $p\le m\le K/2$, so 
$p<K-p$ and similarly $q<K-q$; in particular, $\{p,K-p\}$ and $\{q,K-q\}$ 
each consist of two distinct integers. Moreover, $p\neq q$ implies 
$\{p,K-p\}\neq \{q,K-q\}$. Thus $\mathcal{I}(p,q)$ uniquely determines the 
ordered pair $(p,q)$ up to swapping $a$ and $b$, and the constraint $p<q$ 
removes this ambiguity. Therefore $G(p,q)\cong G(p',q')$ if and only if $(p,q)=(p',q')$.

It follows that
\[
N_3(n)\ge\binom{m}{2}=\binom{\lfloor K/2\rfloor}{2}.
\]

We now establish the claimed bound by considering two cases.

Since $n=9+3K$, we have $(n-15)^2/72=(3K-6)^2/72=(K-2)^2/8$.

\textbf{Case 1:} $K=2t$ for some integer $t\ge 3$.
Then $m=\lfloor K/2\rfloor=t$ and
\[
\binom{m}{2}=\frac{t(t-1)}{2}\ge \frac{(t-1)^2}{2}
\quad\text{(using $t\ge t-1$)}
\]
\[
=\frac{(2t-2)^2}{8}=\frac{(K-2)^2}{8}
\ge \left\lfloor \frac{(K-2)^2}{8}\right\rfloor.
\]

\textbf{Case 2:} $K=2t+1$ for some integer $t\ge 2$.
Then $m=t$ and
\[
\frac{(K-2)^2}{8}=\frac{(2t-1)^2}{8}=\frac{4t^2-4t+1}{8}=\frac{t(t-1)}{2}+\frac{1}{8},
\]
so $\lfloor (K-2)^2/8\rfloor=t(t-1)/2=\binom{m}{2}$.

In both cases,
\[
N_3(n)\ge\binom{\lfloor K/2\rfloor}{2}\ge\left\lfloor \frac{(K-2)^2}{8}\right\rfloor
=\left\lfloor \frac{(n-15)^2}{72}\right\rfloor.
\]

The quadratic-growth statement follows by taking any $c<1/72$ and restricting to $n$ large enough so that $(n-15)^2/72\ge cn^2$.
\end{proof}
\begin{remark}
Table~\ref{Tab2} lists the exact numbers of non-isomorphic strong tricentral $2$-trees
with tail set $\{2,3\}$ for $3\le n\le 12$.
Although Theorem~\ref{thm:quadratic-lower} applies for $n\ge 24$ with $3\mid n$,
the data in Table~\ref{Tab2} already exhibits the onset of rapid growth in the tricentral
regime (for example at $n=11,12$), consistent with the quadratic lower-bound phenomenon previously above.
\end{remark}
\begin{table}[h]
\caption{Enumerative of strong tricentral 2-trees with tail set $\{2,3\}$ for $3 \leq n \leq 12$}
\vspace{.1 in}

\label{Tab2}
\centering
\small
\begin{tabular}{cccccc}
\toprule
$n$ & $\Delta$ & degree sequence & $x$ & $y$ & non-iso count \\
\midrule
3 & 2 & $(2,2,2)$ & 0 & 0 & 1 \\
\midrule
6 & 4 & $(4,4,4,2,2,2)$ & 0 & 3 & 1 \\
\midrule
7 & 4 & $(4,4,4,3,3,2,2)$ & 2 & 2 & 1 \\
\midrule
8 & 5 & $(5,5,5,3,2,2,2,2)$ & 1 & 4 & 1 \\
\midrule
9 & 5 & $(5,5,5,3,3,3,2,2,2)$ & 3 & 3 & 2 \\
9 & 6 & $(6,6,6,2,2,2,2,2,2)$ & 0 & 6 & 1 \\
\midrule
10 & 6 & $(6,6,6,3,3,2,2,2,2,2)$ & 2 & 5 & 4 \\
\midrule
11 & 6 & $(6,6,6,3,3,3,3,2,2,2,2)$ & 4 & 4 & 7 \\
11 & 7 & $(7,7,7,3,2,2,2,2,2,2,2)$ & 1 & 7 & 1 \\
\midrule
12 & 6 & $(6,6,6,3,3,3,3,3,3,2,2,2)$ & 6 & 3 & 2 \\
12 & 7 & $(7,7,7,3,3,3,2,2,2,2,2,2)$ & 3 & 6 & 10 \\
12 & 8 & $(8,8,8,2,2,2,2,2,2,2,2,2)$ & 0 & 9 & 1 \\
\bottomrule
\end{tabular}
\end{table}

The preceding sections reveal a clear progression from rigidity to combinatorial growth as the size of the central core increases.
We conclude by summarizing these transitions and outlining several directions in which the present framework may be extended.

\section{Conclusion and open problems}
We have given a complete structural analysis of strong $r$-central $2$-trees with tail degrees in $\{2,3\}$ for $r\in\{1,2,3\}$.
In the unicentral case, maximal rigidity occurs, with a unique realization for each $n$.
The bicentral case exhibits an intermediate regime, where uniqueness holds exactly for $x\in\{0,2\}$ and breaks down for larger values of $x$, while all feasible maximum degrees are realized.
In the tricentral case, combinatorial flexibility dominates: extremal configurations are unique, but the number of non-isomorphic realizations grows at least quadratically for infinitely many $n$.

Several directions remain open.
A natural problem is to determine the exact asymptotic growth rate of the number of non-isomorphic tricentral $2$-trees and to remove the congruence restrictions in the current quadratic lower bound.
It would also be of interest to characterize uniqueness and growth phenomena for $r$-central $2$-trees with $r\ge 4$, or under more general tail degree constraints.
Finally, exploring extremal and algorithmic properties within these classes, such as diameter, chromatic number, or Wiener index, may further illuminate the interaction between centrality, degree restrictions, and treewidth.

\end{document}